\documentclass[letterpaper,11pt,reqno]{amsart}
\usepackage[utf8]{inputenc}
\usepackage{amsmath}
\usepackage{amssymb}
\usepackage{amsthm}
\usepackage{a4wide}
\usepackage{hyperref}
\usepackage[arrow, matrix, curve]{xy}
\usepackage{xcolor}
\usepackage{enumitem}
\DeclareMathOperator{\GL}{GL}

\DeclareMathOperator{\SL}{SL}
\let\sl\relax
\DeclareMathOperator{\sl}{\mathfrak{sl}}

\DeclareMathOperator{\SU}{SU}
\DeclareMathOperator{\su}{\mathfrak{su}}

\newcommand{\CC}{\mathbb{C}}

\newcommand{\RR}{\mathbb{R}}

\newcommand{\calD}{\mathcal{D}}

\newcommand{\calF}{\mathcal{F}}

\newcommand{\calH}{\mathcal{H}}

\newcommand{\calO}{\mathcal{O}}
\newcommand{\cO}{\mathcal{O}}

\newcommand{\calV}{\mathcal{V}}

\DeclareMathOperator{\Ind}{Ind}

\DeclareMathOperator{\ad}{ad}
\DeclareMathOperator{\Ad}{Ad}

\DeclareMathOperator{\Hom}{Hom}

\DeclareMathOperator{\id}{id}

\DeclareMathOperator{\diag}{diag}

\newcommand{\ip}[2]{\langle #1,#2 \rangle}

\theoremstyle{plain}
\newtheorem{theorem}{Theorem}[section]
\newtheorem{proposition}[theorem]{Proposition}
\newtheorem{lemma}[theorem]{Lemma}
\newtheorem{corollary}[theorem]{Corollary}

\newtheorem{thmalph}{Theorem}

\theoremstyle{definition}

\newtheorem{example}[theorem]{Example}

\newtheorem{remark}[theorem]{Remark}

\newcommand{\fa}{\mathfrak{a}}
\newcommand{\fc}{\mathfrak{c}}
\newcommand{\fg}{\mathfrak{g}}
\newcommand{\fk}{\mathfrak{k}}
\newcommand{\fl}{\mathfrak{l}}
\newcommand{\fm}{\mathfrak{m}}
\newcommand{\fn}{\mathfrak{n}}
\newcommand{\fp}{\mathfrak{p}}
\newcommand{\ft}{\mathfrak{t}}
\newcommand{\fz}{\mathfrak{z}}
\newcommand{\C}{\mathbb{C}}

\newcommand{\R}{\mathbb{R}}

\newcommand{\Dgt}{\Delta (\fg_\C,\ft_\C)}
\newcommand{\Dgtp}{\Delta^+ (\fg_\C,\ft_\C)}

\newcommand{\Dn}{\Delta_n}
\newcommand{\Dnp}{\Delta_n^+}
\newcommand{\Dc}{\Delta_c}
\newcommand{\Dcp}{\Delta_c^+}
\newcommand{\fsl}{\mathfrak{sl}}
\newcommand{\ds}{\displaystyle}

\newcommand{\gs}{\sigma}

\renewcommand{\:}{\, : \,}
\usepackage{sidenotes}
\numberwithin{equation}{section}

\title[Hol. discrete ser. in the gen. Whittaker Plancherel formula II]{The holomorphic discrete series contribution to the generalized Whittaker Plancherel formula II. Non-tube type groups}

\author{Jan Frahm}
\address{Department of Mathematics, Aarhus University, Ny Munkegade 118, 8000 Aarhus C, Denmark}
\email{frahm@math.au.dk}

\author{Gestur \'{O}lafsson}
\address{Department of Mathematics, Louisiana State University, Baton Rouge, LA 70803, USA}
\email{olafsson@math.lsu.edu}

\author{Bent {\O}rsted}
\address{Department of Mathematics, Aarhus University, Ny Munkegade 118, 8000 Aarhus C, Denmark}
\email{orsted@math.au.dk}

\begin{document}



\thanks{Part of the research in this paper was carried out within the online research community on Representation Theory and Noncommutative Geometry sponsored by the American Institute of Mathematics.
The first author was partially supported by a research grant from the Villum Foundation (Grant No. 00025373).
The second author was partially supported by Simons grant 586106.}

\dedicatory{Dedicated to the memory of Gerrit van Dijk}

\maketitle

\begin{abstract}
For every simple Hermitian Lie group $G$, we consider a certain maximal parabolic subgroup whose unipotent radical $N$ is either abelian (if $G$ is of tube type) or two-step nilpotent (if $G$ is of non-tube type). By the generalized Whittaker Plancherel formula we mean the Plancherel decomposition of $L^2(G/N,\omega)$, the space of square-integrable sections of the homogeneous vector bundle over $G/N$ associated with an irreducible unitary representation $\omega$ of $N$. Assuming that the central character of $\omega$ is contained in a certain cone, we construct embeddings of all holomorphic discrete series representations of $G$ into $L^2(G/N,\omega)$ and show that the multiplicities are equal to the dimensions of the lowest $K$-types.\\
The construction is in terms of a kernel function which can be explicitly defined using certain projections inside a complexification of $G$. This kernel function carries all information about the holomorphic discrete series embedding, the lowest $K$-type as functions on $G/N$, as well as the associated Whittaker vectors.
\end{abstract}


\section*{Introduction}

Gerrit van Dijk worked on several aspects of representation theory and harmonic analysis, in particular using special functions and distributions on Lie groups. With his younger colleagues he developed the notion of canonical representations due to Gelfand et al.~as concrete models of unitary representations; giving the explicit decomposition of these into irreducibles amounts in effect to calculating tensor products of unitary irreducible representations or the restriction of such to symmetric subgroups. The groups in question include the classical rank one semisimple Lie groups, such as $G=\SU(n,1)$ with the corresponding Riemannian symmetric space $\calD=G/K$ of Hermitian type, namely the open unit ball in complex Euclidian $n$-space; here the boundary sphere is also the Shilov boundary and of the form $G/P$ for a parabolic subgroup $P$ with nilradical $N$ a Heisenberg group if $n>1$.

In the present paper we shall consider similar groups of Hermitian type, non-tube type, and various models of unitary representations. Even though unitary representations are often considered up to unitary equivalence, for certain problems it is convenient to work with concrete models; this is the case for the above mentioned problems studied by Gerrit van Dijk, namely branching laws, i.e. finding the explicit decomposition of a given irreducible representation restricted to a subgroup. We shall also use some of the tools in his study, namely reproducing kernel Hilbert spaces, interwining operators given as kernel operators, and explicit distribution vectors in unitary representations.

With notation as in the example above, we shall find explicit irreducible invariant subspaces of $L^2(G/N, \omega)$ where $\omega$ is an infinite-dimensional irreducible unitary representation of $N$ defining a homogeneous vector bundle over $X=G/N$; these will be equivalent to holomorphic discrete series representations of $G$. In effect this gives a so-called Whittaker model of these representations corresponding to the Siegel parabolic subgroup $P$ of potential use in the study of automorphic representations and number theory. Thus, our aim in this paper is to
\begin{itemize}
	\item Find explicit Whittaker distribution vectors in the holomorphic discrete series for $G$ with values in $\omega$,
	\item Find explicit intertwining kernel operators between the standard model on $\calD=G/K$ and the Whittaker model on $X=G/N$ for these representations,
	\item Find an explicit expression for the lowest $K$-types on $X$,
\end{itemize}
and to understand links between these problems.

Let us state our results in more detail.

\subsection*{The inclusion $G\subseteq P^+K_\C N_\C$}

Let $G$ be a simple Lie group of Hermitian type, i.e. the associated Riemannian symmetric space $\calD=G/K$ carries a $G$-invariant complex structure. Denote by $\fp^+$ the holomorphic tangent space at the base point $o=eK\in\calD$ and by $\fp^-$ the antiholomorphic tangent space. We assume that $G$ is contained as a real form inside a simply-connected complex Lie group $G_\C$. Let $P^\pm=\exp\fp^\pm$ denote the connected subgroups of $G_\C$ with Lie algebras $\fp^\pm$ and write $K_\C$ for the complexification of $K$ in $G_\C$.

For groups of tube type, we constructed in our previous work \cite{FOO24} the holomorphic discrete series contribution to the Plancherel decomposition of $L^2(G/N,\psi)$, where $N$ is the unipotent radical of some maximal parabolic subgroup and $\psi$ a certain non-degenerate character on $N$. One approach to this problem was to show the inclusion $G\subseteq P^+K_\C N_\C$, where $N_\C$ is the connected complexification of $N$ inside $G_\C$. In order to extend these results to non-tube type groups, we first investigate for which maximal parabolic subgroups $P=MAN$ this inclusion holds.

\begin{thmalph}[see Theorem~\ref{thm:pkn1}]\label{thm:IntroA}
	For every simple Hermitian Lie group $G$ there exists a unique standard maximal parabolic subgroup $P=MAN$ such that $G\subseteq P^+K_\C N_\C$.
\end{thmalph}

\subsection*{Discrete spectrum in the generalized Whittaker Plancherel formula}

We fix the parabolic subgroup $P=MAN$ from Theorem~\ref{thm:IntroA}. The unipotent radical $N$ is either abelian (if $\calD$ is of tube type) or two-step nilpotent (if $\calD$ is of non-tube type). We fix an irreducible unitary representation $\omega$ of $N$ whose central character is contained in a certain cone in the Lie algebra of the center of $N$ (see Section~\ref{sec:FockSpaceReps} for details). Note that $\omega$ is a character if $N$ is abelian and infinite-dimensional otherwise. By the generalized Whittaker Plancherel formula we mean the Plancherel formula for the induced representation $\Ind_N^G(\omega)$ which can be realized on $L^2(G/N,\omega)$, the $L^2$-sections of the homogeneous vector bundle over $G/N$ associated with $\omega$. Our main result is the following:

\begin{thmalph}[see Corollary~\ref{cor:HolDScontribution}]\label{thm:IntroB}
	Every holomorphic discrete series representation of $G$ occurs discretely in $L^2(G/N,\omega)$ with multiplicity equal to the dimension of its highest weight space.
\end{thmalph}

Here, the highest weight space of an irreducible unitary representation $(U,\calH)$ is the subspace $\calH^{\fp^+}$ annihilated by $\fp^+$. It forms an irreducible representation $(\pi,\calV_\pi)$ of $K$ which determines $(U,\calH)$  uniquely. We therefore write $(U_\pi,\calH_\pi)$ for the holomorphic discrete series representation with highest weight $\pi$, and we realize it on a space $\calH_\pi$ of holomorphic $\calV_\pi$-valued functions on $\calD$ (see Section~\ref{sec:HoloDS} for details). Note that $\calV_\pi$ is also the lowest $K$-type of $\calH_\pi$.

\subsection*{Lowest $K$-types and Whittaker vectors via kernel functions}

The proof of Theorem~\ref{thm:IntroB} is by constructing explicit embeddings
$$ \calH_\pi\otimes\overline{\calV_\pi}\to L^2(G/N,\omega), \quad F\otimes\overline\eta\mapsto T_{\pi,\eta}F. $$
The maps $T_{\pi,\eta}:\calH_\pi\to L^2(G/N,\omega)$ ($\eta\in\calV_\pi$) are given in terms of their integral kernels.

In view of Theorem~\ref{thm:IntroA}, every element $g\in G$ decomposes into $g=p^+_N(g)k_{\C,N}^+(g)n_\C^+(g)\in P^+K_\C N_\C$ (not uniquely, but all formulas will be independent of the choice of decomposition). Using this decomposition, we define for every $\eta\in\calV_\pi$ a function $\Psi_{\pi,\eta}:G\times\calD\to\Hom(\calF,\calV_\pi)$ by its adjoint:
\begin{multline*}
	\Psi_{\pi,\eta}(x,g\cdot o)^*:\calV_\pi\to\calF,\\
	\xi \mapsto \omega(n_\C^+(g^{-1}x))^{-1}m\Big(\pi(k_{\C,N}^+(g^{-1}x))^{-1}\circ j_\pi(g,o)^{-*}\xi\otimes\overline{\eta}\Big) \qquad (x,g\in G,\xi\in\calV_\pi),
\end{multline*}
where $j_\pi(g,z)\in\GL(\calV_\pi)$ ($g\in G$, $z\in\calD$) is the cocycle used to define the holomorphic discrete series action on $\calD$ (see \eqref{eq:DefJpiAndKpi}) and $m:\calV_\pi\otimes\overline{\calV}_\pi\to\calF$ is a certain matrix coefficient map (see \eqref{eq:DefMatrixCoeffMap}). This construction is discussed in detail in Section~\ref{sec:ExplicitKernels}, but the main point is that the definition of $\Psi$ is completely explicit when using the embedding $G\subseteq P^+K_\C N_\C$.

The operator $T_{\pi,\eta}:\calH_\pi\to L^2(G/N,\omega)$ ($\eta\in\calV_\pi$) is most easily characterized by its adjoint $T_{\pi,\eta}^*:L^2(G/N,\omega)\to\calH_\pi$ which has $\Psi(x,z)$ as integral kernel:
$$ T_{\pi,\eta}^*f(z) = \int_{G/N}\Psi_{\pi,\eta}(x,z)f(x)\,d(xN) \qquad (f\in L^2(G/N,\omega),z\in\calD). $$

The function $\Psi$ carries all information about the embedding of the associated holomorphic discrete series representation into $L^2(G/N,\omega)$ in the following sense:

\begin{thmalph}[see Proposition~\ref{prop:AdjointIntertwiner}, Lemma~\ref{lem:LKTembedding} and Proposition~\ref{prop:WhittakerVectorsAsAntiholFcts}]\label{thm:IntroC}
	\begin{enumerate}
		\item\label{thm:IntroC1} The map
		$$ \calH_\pi\otimes\overline{\calV_\pi}\to L^2(G/N,\omega), \quad F\otimes\overline{\eta}\mapsto T_{\pi,\eta}F $$
		constructs the holomorphic discrete series contribution in Theorem~\ref{thm:IntroB}.
		\item\label{thm:IntroC2} The embedding $T_{\pi,\eta}$ restricted to the lowest $K$-type $\calV_\pi$ is given by
		$$ T_{\pi,\eta}\xi(x) = \Psi_{\pi,\eta}(x,o)^*\xi \qquad (\xi\in\calV_\pi,x\in G). $$
		\item\label{thm:IntroC3} Identifying distribution vectors on $\calH_\pi$ with antiholomorphic $\calV_\pi^\vee$-valued functions on $\calD$ (see Section~\ref{sec:BackToWhittaker} for details), the Whittaker vector $W\in\Hom_N(U_\pi^\infty|_N,\omega)$ associated with the embedding $T_{\pi,\eta}\in\Hom_G(U_\pi,\Ind_N^G(\omega))$ by Frobenius reciprocity is given by the $\calV_\pi^\vee\otimes\calF=\Hom(\calV_\pi,\calF)$-valued antiholomorphic function
		$$ W(z) = \Psi_{\pi,\eta}(e,z)^* \qquad (z\in\calD). $$
	\end{enumerate}
\end{thmalph}

Items \eqref{thm:IntroC2} and \eqref{thm:IntroC3} show that the kernel $\Psi_{\pi,\eta}$ contains both information about the Whittaker vector (when restricted to $\{e\}\times\calD$) and about the lowest $K$-type embedding (when restricted to $G\times\{o\}$).

\subsection*{Relation to other work}

This paper extends parts of the results in our previous work \cite{FOO24} to the case of non-tube type groups. In \cite{FOO24} we also give explicit formulas for Whittaker vectors in different models of the holomorphic discrete series. This should also be possible for non-tube type groups using the $L^2$-model obtained by Ding in \cite{Din99}. Moreover, in \cite{FOO24} we show that the holomorphic discrete series contribution to $L^2(G/N,\omega)$ is given by the boundary values of functions in a Hardy space of holomorphic sections on a complex domain in $G_\C/N_\C$ containing $G/N$ in its boundary. We expect that similar results can be obtained in the more general case using the same methods.

\subsection*{Structure of the paper}

In Section~\ref{sec:BoundedDomains} we recall some basic structure theory for Hermitian Lie groups and the associated bounded symmetric domains. This is used in Section~\ref{sec:Inclusion} to show Theorem~\ref{thm:IntroA}. Section~\ref{sec:HoloDS} is about the standard realization of holomorphic discrete series representations on a space of holomorphic functions on the bounded domain. In Section~\ref{sec:FockSpaceReps} we introduce the irreducible unitary representations $\omega$ of $N$ that define the Whittaker space $L^2(G/N,\omega)$ and realize them on a Fock space of holomorphic functions which are square integrable with respect to a Gaussian measure. Finally, Section~\ref{sec:WhittakerModels} provides the construction of the embeddings $T_{\pi,\eta}$ of the holomorphic discrete series into $L^2(G/N,\omega)$ and the proofs of Theorems~\ref{thm:IntroB} and \ref{thm:IntroC}.

\section{Bounded symmetric domains}\label{sec:BoundedDomains}

In this section we recall some basic standard facts about bounded symmetric domains. For more details we refer the reader to \cite{KW65b,KW65a,Loo77,M64, Sa80,Sch84}.

\subsection{The Harish-Chandra realization of bounded domains}\label{subHC}

In this article we assume that $G$ is a connected simple Lie group with Lie algebra $\fg$. We assume that $G$ is
contained as a closed subgroup in a simply connected complex Lie group $G_\C$ with Lie algebra $\fg_\C$. Let
$\sigma_\fg : \fg_\C\to \fg_\C$ be the conjugation with respect to $\fg$, $\sigma_\fg (X+iY) = X-iY$ for $X,Y\in \fg$.
Then $\sigma_\fg$ can be lifted to a homomorphism $\sigma_G :G_\C\to G_\C$ such that $G\subseteq G_\C^{\sigma_G}$ is open.

Let $\theta_\fg:\fg\to\fg$ be a Cartan involution which we extend to a complex-linear involution $\fg_\C\to\fg_\C$. It lifts to an involution $\theta_G:G_\C\to G_\C$ such that $K=G^{\theta_G}$ is a maximal compact subgroup of $G$ with complexification $K_\C=G_\C^{\theta_G}$. If the role of $G$ or $\fg$ is clear from the context, we simply write $\theta$ for both involutions.

Denote by $\fg = \fk \oplus \fp$ the corresponding Cartan decomposition of $\fg$. We assume that $G/K$ is a bounded symmetric domain. This happens if and only if the center $\fz $ of $\fk$ is non-zero, and in that case it is one-dimensional. There exists an element $z_0\in \fz$, unique up to sign, such 
that $\ad(z_0)$ has eigenvalues $0$, $\pm i$ on $\fg_\C$. The zero-eigenspace is $\fk_\C$. We write
$$ \fp^\pm = \{X\in \fg_\C \: [z_0,X] = \pm iX\}. $$
Then $\fp_\C= \fp^+\oplus \fp^-$ and $\fp^\pm $ are abelian Lie algebras and the corresponding analytic subgroups $P^\pm = \exp \fp^\pm\subseteq G_\C$ are normalized by $K_\C$. The
set $P^+K_\C P^-$ is open and dense in $G_\C$ and the multiplication
map
$$ P^+\times K_\C \times P^-\to G_\C, \quad (p^+,k,p^-)\mapsto p^+k p^- $$
is a holomorphic diffeomorphism into its image. 
For $g\in P^+K_\C P^-$ we write $g = p^+(g)k_\C (g) p^-(g)$.
We have $G\subseteq 
P^+K_\C P^-$,  $G\cap K_\C P^-= K$ and the map
$$ \iota: G/K \to \fp^+, \quad gK \mapsto \log_{\fp^+}(p^+(g)) $$ 
is a well defined $G$-equivariant embedding. Its image $\calD=\iota(G/K)$ is a bounded circular domain in the complex vector space $\fp^+$, the Harish-Chandra realization of $G/K$, and we write
\[ g\cdot Z = \iota(g\exp Z) \qquad (g\in G,Z\in\calD) \]
for the corresponding action of $G$ on $\calD$.

Embedding $\fp^+$ into $G_\C/K_\C P^-$ by $Z\mapsto\exp(Z)K_\C P^-$, we can consider the action of $G_\C$ on $\calD$ and $\fp^+$. Clearly, $P^+$ acts on $\fp^{+}$ by translation $\exp (X)\cdot Z= Z+X$ and $K_\C$ acts linearly by $k\cdot Z = \Ad (k)Z$.

\subsection{Strongly orthogonal roots}\label{subsSOR}

Let $\ft$ be a Cartan subalgebra of $\fg$ containing $z_0$, then $\ft\subseteq \fk$. Let $\Dgt$ be the roots of $\ft_\C$ in $\fg_\C$ and for $\alpha\in\Dgt$ write $\fg_\C^\alpha$ for the corresponding root space. We define subsets
\begin{align*}
\Dc &= \{\alpha \in\Dgt\: \fg_\C^\alpha \subseteq \fk_\C\} = \{\alpha \in \Dgt\: \alpha (z_0)=0\}\\
\Dn &= \{\alpha \in \Dgt \: \fg_\C^\alpha \subseteq \fp_\C\} = \{\alpha\in \Dgt\: \alpha (z_0) \not= 0\}\\
\Dnp & = \{\alpha \in \Dgt \: \fg_\C^\alpha \subseteq \fp^+\} = \{\alpha \in \Dn \: \alpha (z_0)=i\}
\end{align*}
and call the roots in $\Dc$ compact and the roots in $\Dn$ non-compact. Choose a positive system of roots $\Dgtp $ such that $\Dnp \subseteq \Dgtp $. Then $\Dcp := \Dgtp\cap \Dc$ is a positive system in
$\Dc$ and $\Dgtp = \Dnp \cup \Dcp$. 
 
Let $\gamma_1,\ldots , \gamma_r$ be a maximal set of (long) strongly orthogonal roots in $\Dnp$ constructed in the
usual way: $\gamma_1$ is the maximal root in $\Dnp$ and $\gamma_{j+1} $ the maximal root in $\Dnp$ strongly
orthogonal to $\gamma_1,\ldots ,\gamma_j$. Let $h_j\in i\ft$ be the corresponding
co-root vector. In particular, $\gamma_i(h_j) = 2\delta_{ij}$. Let $\fc = \sum_{j=1}^r \R ih_j \subseteq \ft$ and
$\fc_\C$ its complexification. Abusing notation, we also write $\gamma_i$ for its restriction to $\fc_\C$. Moore's Theorem describes the roots of $\fc_\C$ in $\fg_\C$, i.e. the restrictions of roots in $\Delta(\fg_\C,\ft_\C)$ to $\fc_\C$. Here, we call a restricted root compact resp. non-compact if it is the restriction of a compact resp. non-compact root.

\begin{theorem}[Moore's Theorem, first version]
	For the set $\Delta (\fg_\C,\fc_\C)$ of restricted roots there are two possibilities:
	\begin{enumerate}
		\item $\calD$ is of tube type. In this case
		$$ \qquad\ds \Delta (\fg_\C,\fc_\C) =\{\pm\gamma_i:1\leq i\leq r\} \cup \{\tfrac{1}{2}(\pm\gamma_i\pm \gamma_j)\: 1\le i<j \le r\}, $$
		the roots $\pm\gamma_i$ and $\pm\frac{1}{2}(\gamma_i + \gamma_j)$ are non-compact and the roots $\pm\frac{1}{2}(\gamma_i - \gamma_j)$ are compact.
		\item $\calD$ is of non-tube type. In this case
		$$ \qquad\quad \ds \Delta  (\fg_\C,\fc_\C)=\{\pm\tfrac{1}{2}\gamma_i,\pm\gamma_i:1\leq i\leq r\} \cup \{\tfrac{1}{2}(\pm\gamma_i\pm \gamma_j)\: 1\le i<j \le r\}, $$
		the roots $\pm\gamma_i$, $\pm\tfrac{1}{2}\gamma_i$ and $\pm\frac{1}{2}(\gamma_i + \gamma_j)$ are non-compact and the roots $\pm\tfrac{1}{2}\gamma_i$ and $\pm\frac{1}{2}(\gamma_i - \gamma_j)$ are compact.
	\end{enumerate}
\end{theorem}

\subsection{The $\su (1,1)$-embeddings}\label{SuSSU(11)}

Let for the moment
\[G =
\SU (1,1) = \left\{ \begin{pmatrix} a & b\\ \bar b & \bar a\end{pmatrix}\: |a|^2 - |b|^2=1\right\} \subseteq \SL(2,\C)\]
with Lie algebra 
\[ \fg = \su (1,1) =\left\{\begin{pmatrix} ix & z \\ \bar z & - ix\end{pmatrix}\: x\in \R, z\in \C\right\}\subseteq \sl(2,\C).\]
We have $\gs_G  (g) = I_{1,1} g^{-*} I_{1,1}$ and
$\sigma_\fg (X) = - I_{1,1} X^*I_{1,1}$, with the notation $I_{1,1} = \diag ( 1, -1)$ and $X^*=\overline{X}^\top$. We choose $\theta_G(g) = I_{1,1}gI_{1,1}$ as a Cartan involution, then $\theta_{\fg} (X) = I_{1,1}XI_{1,1}$, $\fk = \{\diag( ix,-ix)\: x\in \R\}$ and $K=\{\diag(\gamma ,\gamma^{-1})\: |\gamma | =1\}$.  

Let
\[z_0 =\begin{pmatrix} \frac{i}{2} & 0 \\ 0 & -\frac{i}{2} \end{pmatrix} ,
\quad 
e_0 =
\begin{pmatrix} 0  & 1\\ 0 & 0 \end{pmatrix}, \quad 
f_0=\begin{pmatrix} 0 & 0 \\ 1 & 0  \end{pmatrix}, \]
then
\[ \fk_\C = \C z_0,\quad \fp^+ = \C e_0, \quad \fp^- = \C f_0 \]
and
\[ h_0 = [e_0,f_0] = \begin{pmatrix} 1 & 0\\ 0 & -1 \end{pmatrix} = -2iz_0 .\]
We have
\[P^+K_\C P^- = \left\{g=\begin{pmatrix} a & b \\ c & d\end{pmatrix} : ad - b c =1, d\not= 0\right\}\]
Identifying $K_\C \simeq \C^*$ and $\fp^+=\C e_0\simeq \C$ we have for a matrix $g$ as above
\begin{equation}
	k_\C (g) = \frac{1}{d}\quad \text{and} \quad p^+(g) = \frac{b}{d}.\label{eq:KandP+ProjectionsSL2}
\end{equation}
Therefore, the Harish-Chandra embedding is given by
\[ \iota:G/K\to\fp^+, \quad \begin{pmatrix} a & b \\ \bar b & \bar a\end{pmatrix} K \mapsto \frac{b}{\bar a} \in \calD_0=\{z\in \C \: |z| <1\}\]
and the corresponding action is the usual action by fractional linear transformations:
\[g\cdot z = \frac{az + b}{cz + d} \qquad \mbox{for }g=\begin{pmatrix}a&b\\c&d\end{pmatrix},z\in\calD_0.\]

Now to the general case.  For every $1\leq j\leq r$ the root spaces $\fg_\C^{\pm\gamma_j}$ are one-dimensional and we choose $e_j\in\fg_\C^{ \gamma_j}$, so that with $f_j=\sigma_\fg (e_j) \in\fg_\C^{-\gamma_j}$ we have $[e_j,f_j]=h_j$ with $h_j$ as in Section~\ref{subHC}. Let $x_j = e_j + f_j\in \fg$.

 For $j=1,\ldots ,r$ define   $\varphi_j :\fsl (2,\C) \to \fg_\C$   by
\[h_0 =\begin{pmatrix} 1 & 0 \\ 0 & -1\end{pmatrix}\mapsto  h_j
,\quad e_0 = \begin{pmatrix} 0 & 1\\ 0 & 0 \end{pmatrix} \mapsto e_j
\quad\text{and} \quad f_0= \begin{pmatrix} 0 & 0 \\ 1 & 0 \end{pmatrix} \mapsto f_j .\]
Then $\varphi_j$ is a Lie algebra homomorphism such that $\varphi_j\circ \theta_{\su (1,1)}
= \theta_\fg \circ \varphi_j$ and $\varphi_j\circ \sigma_{\su (1,1)}
= \sigma_\fg \circ \varphi_j$. As $\SL (2,\C)$ is simply connected, there exists a unique homomorphism, which we
will denote by the same letter, $\varphi_j : \SL (2,\C)\to G_\C$ such that 
$\varphi_j (\exp X) = \exp (\varphi_j (X))$ and 
$\varphi_j\circ \theta_{\SU (1,1)} = \theta_G \circ \varphi_j$ as well as $\varphi_j\circ \sigma_{\SU (1,1)} = \sigma_G \circ \varphi_j$.
It follows in particular that $\varphi_j (\SU (1,1)) \subseteq G$.
As the roots $\gamma_1,\ldots \gamma_r$ are strongly orthogonal it follows that the algebras $\fg_\C^j = \varphi_j (\fsl (2,\C))$ commute pairwise.

\begin{lemma}
The subspace $\fa = \bigoplus_{j=1}^r\R x_j$ is maximal abelian in $\fp$.
\end{lemma}

For $\SU (1,1)$ we let 
\[c_0 =\exp \left( -\frac{\pi }{4} (e_0-f_0)\right)  = \frac{1}{\sqrt{2}} \begin{pmatrix}1 & -1\\ 1 & 1\end{pmatrix}\in \SL (2,\R)\]
and
\[c_j = \varphi_{j} (c_0) = \exp \left( -\frac{\pi }{4} (e_j-f_j)\right) .\]
Then simple matrix multiplication shows that 
$ \Ad (c_0)  h_0 = e_0 + f_0 = x_0$ and hence $\Ad (c_j) h_j = x_j$. Let
$c=c_1\cdots c_r$. Then we have 

\begin{lemma} $\Ad (c) \fc = \fa$.\end{lemma} 

Define $\lambda_j =\gamma_j\circ \Ad (c)^{-1} \in \fa^* $.
Then Moore's Theorem implies:  

\begin{theorem}[Moore's Theorem, second version]\label{thm:Moore2}
	For the set $\Delta (\fg,\fa)$ of restricted roots there are two possibilities:
	\begin{enumerate}
		\item $\calD$ is of tube type. In this case
		$$ \qquad\ds \Delta (\fg,\fa) =\{\pm\lambda_i:1\leq i\leq r\} \cup \{\tfrac{1}{2}(\pm\lambda_i\pm \lambda_j)\: 1\le i<j \le r\}. $$
		\item $\calD$ is of non-tube type. In this case
		$$ \qquad\quad \ds \Delta  (\fg,\fa)=\{\pm\tfrac{1}{2}\lambda_i,\pm\lambda_i:1\leq i\leq r\} \cup \{\tfrac{1}{2}(\pm\lambda_i\pm \lambda_j)\: 1\le i<j \le r\}. $$
	\end{enumerate}
\end{theorem}

In the case where $\calD$ is of non-tube type, the root spaces $\fg^{\frac{1}{2}\lambda_j}$ have a complex structure. For this, let
\begin{equation}
	E = \sum_{j=1}^r\varphi_j\begin{pmatrix}i&-i\\i&-i\end{pmatrix} \in \bigoplus_{j=1}^r\fg^{\lambda_j}, \qquad \mbox{and} \qquad F = \theta(E) = \sum_{j=1}^r\varphi_j\begin{pmatrix}i&i\\-i&-i\end{pmatrix} \in \bigoplus_{j=1}^r\fg^{-\lambda_j}.\label{eq:DefE}
\end{equation}

\begin{lemma}[{see \cite[Lemma 2.2.4]{VR76} and \cite[Chapter III, \S2 and \S3]{Sa80}}]\label{lem:ComplexStructure}
	The map $J:=\ad(E)\circ\theta=\theta\circ\ad(F)$ on $\bigoplus_{j=1}^r\fg^{\frac{1}{2}\lambda_j}$ has the following properties:
	\begin{enumerate}
		\item $J^2=-\id$,
		\item $J(\fg^{\frac{1}{2}\lambda_j})=\fg^{\frac{1}{2}\lambda_j}$ for all $1\leq j\leq r$,
		\item For every $1\leq j\leq r$, the complexified root space $\fg^{\frac{1}{2}\lambda_j}_\C$ decomposes as
		$$ \fg^{\frac{1}{2}\lambda_j}_\C = \fg^{\frac{1}{2}\lambda_j,+}_\C\oplus\fg^{\frac{1}{2}\lambda_j,-}_\C, $$
		where
		$$ \fg^{\frac{1}{2}\lambda_j,\pm}_\C = \fg_\C^{\frac{1}{2}\lambda_j}\cap(\fk_\C\oplus\fp^\pm) = \{Z\in\fg_\C^{\frac{1}{2}\lambda_j}:JZ=\pm iZ\}. $$
	\end{enumerate}
\end{lemma}

\section{The inclusion $G\subseteq P^+K_\C N_\C$}\label{sec:Inclusion}

In this section we show that there is exactly one choice of a maximal parabolic subgroup $P =LN \subseteq G$
such that $G\subseteq P^+ K_\C N_\C$. This embedding will be fundamental in our construction of Whittaker vectors in Section~\ref{sec:ExplicitKernels}. For tube type domains, this inclusion was obtained earlier in \cite{FOO24}. The proof uses $\SU (1,1)$-reduction by first proving the inclusion for $G=\SU (1,1)$ and then using the commuting homomorphisms $\varphi_j$, $j=1,\ldots ,r$, to prove it for
the general case. The proof further provides us with the explicit decomposition of elements in a maximally split torus which turns out to be important to prove the $L^2$-condition in Theorem~\ref{thm:EmbeddingsAreL2}.

\subsection{The case $G=\SU(1,1)$}

Let $G=\SU(1,1)\subseteq \SL (2,\C)$. We use the notation from the last section and set $x_0 = e_0+f_0$ and
$\fa = \R x_0$. Then we define
\begin{align*}
	A &= \exp \R x_0=  \{a_t:t\in\RR\}, & a_t &= \exp t x_0   =
	 \begin{pmatrix} \cosh t & \sinh t\\ \sinh t &\cosh t\end{pmatrix},\\
	N &=  \{n_x:x\in\RR\}, & n_x &= \exp x\begin{pmatrix} i & -i \\ i & -i\end{pmatrix} = 
	\begin{pmatrix} 1+ ix & -ix \\ ix & 1-ix\end{pmatrix}.
\end{align*}
The natural complexification of $N$ inside $G_\C=\SL(2,\C)$ is
$$ N_\CC = \{n_z:z\in\CC\}, \qquad n_z = \begin{pmatrix} 1+ iz & -iz \\ iz & 1-iz\end{pmatrix}. $$

\begin{theorem}[{see \cite[Lemma 4.1 and Theorem 4.2]{FOO24}}]\label{lem:NKPdecompositionSU11}
A matrix
$$ g=\begin{pmatrix}a&b\\c&d\end{pmatrix}\in\SL(2,\CC) $$
is contained in $P^+ K_\CC N_\CC$ if and only if $c+d\neq 0$. More precisely, writing $p_z^+ = \exp z e_0$ and $k_\theta = \exp \theta z_0$, we have $g=p_z^+ k_\theta n_w$ with
\[ e^{i\theta/2} = \frac{1}{c+d}, \qquad w = \frac{-ic}{c+d}\qquad\text{and}\qquad z=
\frac{(a+b)(c+d)-1}{(c+d)^2}. \]
In particular, $G\subseteq P^+ K_\C N_\C$ and $ P^+ K_\CC N_\CC$ is dense in $G_\CC$.
\end{theorem}

\begin{corollary}
	For every $t\in\R$ we have $a_t=p^+_zk_{2it}n_w$ with
	$$ z = 1-e^{-2t} \qquad \mbox{and} \qquad w = \frac{1}{2i}(1-e^{-2t}). $$
\end{corollary}

\subsection{The general case}\label{sec:P+KNgeneralCase}

We now prove Theorem \ref{lem:NKPdecompositionSU11} for the general case. We
use the notation from Sections \ref{subsSOR} and \ref{SuSSU(11)}.

For $j=1,\ldots,r$ let $y_j = x_1+\cdots + x_j$. Then $y_j$ determines a maximal parabolic subalgebra
\[\fp_j = \fl_j \oplus \fn_j = \fm_j \oplus \R y_j \oplus \fn_j\]
in $\fg$ given by
\[\fl_j = \fz_{\fg} (y_j)\quad\text{and}\quad \fn_j  = \bigoplus_{\alpha (y_j)>0}\fg^\alpha . \] 
In particular, it follows from Theorem~\ref{thm:Moore2} that $\fg^{\lambda_k}\subseteq\fn_j$ if and only if $k\leq j$. Denote by $P_j=L_jN_j$ the corresponding maximal parabolic subgroup of $G$.

\begin{theorem}\label{thm:pkn1} Let $j= 1,\ldots ,r$, then  
$G\subseteq P^+ K_\C N_{j\C}$ if and only if $j=r$.
\end{theorem} 

\begin{proof} Write $P=P_j$ and $L=L_j$, $N=N_j$ for short. The set $P^+ K_\C N_{\C}$ is left $K_\C$-invariant and right $(K\cap L)_\C
N_\C$-invariant. We have $G=KLN$ and $L= (K\cap L)A(K\cap L)$, so
\[ G = K A  (K\cap L) N .\]
It therefore suffice to show that $A \subseteq P^+K_\C N_\C $ if and only if $j=r$. As $\fa = \bigoplus_{k=1}^r \R x_k$, we have
$A\subseteq P^+K_\C N_\C$ if and only if $\exp t x_k\in P^+K_\C N_\C$ for all $t\in\R$ and $k=1,\ldots ,r$.   But $\SU (1,1)$-reduction shows that
\[\exp tx_k = \exp ((1-e^{-2t})e_k) \exp (-th_k) \exp \left(\frac{1}{2i}(1-e^{-2t})\varphi_k \begin{pmatrix} i & - i\\ i & - i \end{pmatrix}\right) .\]
By definition we have
\[\varphi_k \begin{pmatrix} i & - i\\ i & - i \end{pmatrix}\in \fg^{\lambda_k}.\]
Thus, $\exp tx_k \in N_\C$ for all $t\in\R$ if and only if $\fg^{\lambda_k}\subseteq \fn$. By the considerations above this is true for all $1\leq k\leq r$ if and only if $j=r$.
\end{proof}

From now on we write $N_\C$ for $N_{r\C}$.

\begin{remark}\label{rem:P-Decomposition}
	Note that by applying $\sigma_G$ to the decomposition $G\subseteq P^+K_\C N_\C$ and using that $\sigma_G (P^+) = P^-$,
	$\sigma_G(K_\C) = K_\C$ and $ \sigma_G(N_\C)=N_\C$,
	it also follows that $G\subseteq P^-K_\C N_\C$. 
\end{remark}

In contrast to the case where $\calD$ is of tube type, the multiplication map $P^+\times K_\C\times N_\C\to G_\C$ is not bijective. In fact, by Lemma~\ref{lem:ComplexStructure} the intersection $K_\C P^+\cap N_\C$ is non-trivial if and only if $\calD$ is of non-tube type, and in this case
$$ K_\C P^+\cap N_\C = \exp\Bigg(\bigoplus_{j=1}^r\fg_\C^{\frac{1}{2}\lambda_j,+}\Bigg). $$
Taking this intersection into account, we obtain:

\begin{corollary} The set $P^+K_\C N_\C$ is open in $G_\C$ and the multiplication map
$$ (P^+\times K_\C) \times_{K_\C P^+\cap N_\C} N_\C \to P^+ K_\C N_\C, \quad [p^+,k,n]\mapsto p^+kn $$
is biholomorphic.
\end{corollary}

\begin{remark}\label{rem:DisambiguityDecomposition}
	For $g\in P^+ K_\C N_\C$ write
	$$ g = p^+_N (g) k^+_{\C ,N}(g) n^+_\C (g) \in P^+ K_\C N_\C $$
	for some choice of decomposition. Whenever using this notation, the expression under consideration will be independent of the chosen representatives.\\
	Further note that if we decompose
	$$ g = p^-_N(g) k^-_{\C, N}(g) n^-_\C(g) \in P^- K_\C N_\C, $$
	then applying $\sigma$ yields
	\begin{equation}
		p^+_N(g) = \sigma(p^-_N(g)), \qquad k^+_{\C,N}(g) = \sigma(k^-_{\C,N}(g)), \qquad n^+_\C(g) = \sigma(n^-_\C(g)).\label{eq:Decomposition+vs-}
	\end{equation}
\end{remark}

From the explicit calculation in the proof of Theorem \ref{thm:pkn1} we further obtain:

\begin{corollary}\label{thm:pkn2}
For $j=1,\ldots ,r$ and $t\in\R$ we have
\begin{align}\label{eq:pkn}
 k_{\C,N}(\exp tx_j) &= \exp(-th_j),\\
 n_\CC(\exp tx_j) &= \exp\left(\frac{1}{2i}(1-e^{-2t})\varphi_j  \begin{pmatrix} i & - i\\ i & - i \end{pmatrix}\right).
\end{align}
\end{corollary}

\section{The holomorphic discrete series}\label{sec:HoloDS}

In this section we recall the realization of holomorphic discrete series representations on spaces of holomorphic functions on the bounded domain $\calD$. We refer to \cite{N00, O00, VR76} for more details.

We start by defining the universal {\it cocycle and kernel}, see \cite[pp 64]{Sa80}. For
that we let $g\in G_\C $ and $z,w \in \fp^+$. Then, whenever defined, which in particular is true for $g\in G$ and $z,w\in\calD$, we set, using the notation from Section \ref{subHC}:
\begin{itemize}
\item[\rm (J)] (\emph{universal cocycle}) $J(g,z) :=   k_\C (g\exp z)  $,
\item[\rm (K)] (\emph{universal kernel}) $K(z,w) = k_\C (\exp (- \sigma_\fg (w))\exp (z))$.
\end{itemize}
The following relations are easy to verify:

\begin{lemma}\label{lem:IdentitiesCocycleKernel}
	For $k\in K_\C$ define $k^* = \sigma_G (k)^{-1}$. Then the following holds true for $a,b\in G_\C$, $z,w\in \fp^+$ and $k\in K_\C$:
	\begin{enumerate}
		\item\label{lem:IdentitiesCocycleKernel1} $J(k,z) = k$,
		\item\label{lem:IdentitiesCocycleKernel2} $J(ab,z) = J (a,b\cdot z) J(b,z)$,
		\item\label{lem:IdentitiesCocycleKernel3} $K(z,o)=K(o,w)=e$,
		\item\label{lem:IdentitiesCocycleKernel4} $K(z,w) = K(w,z)^*$.
	\end{enumerate}
	Moreover, if $a\in G$ we further have
	\begin{enumerate}
		\setcounter{enumi}{4}
		\item\label{lem:IdentitiesCocycleKernel5} $J(a,w)^* K(a\cdot z,a\cdot w) J (a,z) = K(z,w)$,
		\item\label{lem:IdentitiesCocycleKernel6} $K(a\cdot o,a\cdot o)  = J(a,o)^{-*}J(a,o)^{-1}$.
	\end{enumerate}
\end{lemma}

Now fix a finite-dimensional holomorphic representation $(\pi,\calV_\pi)$ of $K_\C$. We endow $\calV_\pi$ with an inner product that makes $\pi$ unitary, so that
\begin{equation}
	\pi(k)^* = \pi(k^*) \qquad (k\in K_\C).\label{eq:PiAdjoint}
\end{equation}
Let
\begin{equation}
	j_\pi(g,z) = \pi(J(g,z))^{-1} \qquad \mbox{and} \qquad K_\pi(z,w) = \pi(K(z,w))^{-1},\label{eq:DefJpiAndKpi}
\end{equation}
whenever $g\in G$ and $z,w\in\calD$. Then by Lemma~\ref{lem:IdentitiesCocycleKernel}~\eqref{lem:IdentitiesCocycleKernel2}, the formula
\[U_\pi (g) F(z) = j_\pi (g^{-1},z) F(g^{-1}z) \qquad (g\in G,F\in\calO(\calD,\calV_\pi),z\in\calD)\] 
defines a representation of $G$ on $\cO (\calD,\calV_\pi)$, the space of $\calV_\pi$-valued holomorphic functions on $\calD$. By Lemma~\ref{lem:IdentitiesCocycleKernel}~\eqref{lem:IdentitiesCocycleKernel5} it preserves the inner product
$$ \langle F_1,F_2\rangle_\pi = \int_{\calD}\langle K_\pi(z,z)^{-1}F_1(z),F_2(z)\rangle\,d^*z, $$
where $d^*z$ denotes a $G$-invariant measure on $\calD$. This defines a unitary representation $(U_\pi,\calH_\pi)$, but $\calH_\pi$ could be zero. A necessary and sufficient condition for $\calH_\pi\neq\{0\}$ can be given in terms of the highest weight $\mu_\pi\in\ft_\C^*$ of $\pi$:

\begin{theorem}[Harish-Chandra]
	Let $\rho = \frac{1}{2}\sum_{\alpha\in \Delta^+(\fg_\C,\ft_\C)} \alpha$. Then $\calH_\pi\neq\{0\}$ if and only if
	\begin{equation}\label{eq:HCcond}
		\ip{\mu + \rho}{\alpha} < 0,\quad \text{for all } \alpha \in\Dnp.
	\end{equation}
	In this case, $(U_\pi,\calH_\pi)$ is irreducible and belongs to the discrete series.
\end{theorem}

The representations $(U_\pi,\calH_\pi)$ are called \emph{holomorphic discrete series representations}. Moreover, the Hilbert space $\calH_\pi$ is a reproducing kernel Hilbert space with reproducing kernel $K_\pi(z,w)$ (after suitable normalization of the measure $d^*z$).

Note that the condition \eqref{eq:HCcond} only depends on the restriction of $\mu_\pi$ to $\fc_\C$. Writing $\mu_\pi|_{\fc_\C}=-\frac{1}{2}\sum_{j=1}^r\mu_{\pi,j}\gamma_j$ and $\rho|_{\fc_\C}=\frac{1}{2}\sum_{j=1}^r\rho_j$, we find that \eqref{eq:HCcond} is equivalent to
\begin{equation}
	\mu_{\pi,j}\geq\rho_j \qquad \mbox{for all }1\leq j\leq r.\label{eq:DSconditionMuJ}
\end{equation}
In this notation, the action of the torus $\exp\fc_\C$ on a highest weight vector $\xi\in\calV_\pi$ is given by
\begin{equation}
	\exp\left(\sum_{j=1}^rt_jh_j\right)\xi = \left(\prod_{j=1}^r e^{-\mu_{\pi,j}t_j}\right)\xi.\label{eq:ActionHighestWeightVector}
\end{equation}

\begin{example}[The case of $G=\SU(1,1)$]
We identify $K_\C$ with $\C^*$ by $k_\theta\mapsto e^\theta$ and $\fp^+$ with $\C$ by $ze_0\mapsto e_0$. For
$$g=\begin{pmatrix} a & b \\ c & d\end{pmatrix} \in \SL (2,\C)$$ 
and $z, w\in \C$, $\gamma\in \C^*$ we have
\[ g\exp z e_0 = \begin{pmatrix} a & b \\ c & d\end{pmatrix} \begin{pmatrix} 1 & z \\ 0 & 1\end{pmatrix}
= \begin{pmatrix} a & az + b\\ c & cz + d\end{pmatrix}.\]
Thus $J(g,z) = (cz + d)^{-1}$ by \eqref{eq:KandP+ProjectionsSL2}.
Moreover, for $z,w\in\C$ we find 
\[\exp (-\bar w f_0)\exp (z e_0) =\begin{pmatrix} 1 & 0 \\ -\bar w & 1\end{pmatrix}\begin{pmatrix} 1 & z \\ 0 & 1\end{pmatrix}
\begin{pmatrix} 1 & z \\ -\bar w & 1- z\bar w\end{pmatrix},\]
so, again by \eqref{eq:KandP+ProjectionsSL2}, we have $K(z,w) = (1-z\bar w)^{-1}$.

The finite-dimensional irreducible representation of $\widetilde K\simeq \R$ are the characters $t\mapsto e^{i\lambda t}$.
We then have
\[j_\lambda (g,z) = (cz +d)^{-\lambda}\quad \text{and}\quad K_\lambda (z,w) = (1 - z\bar w)^{-\lambda} .\]
Note that 
\[ j_\lambda (a_t,z)= (\sinh (t)z + \cosh (t))^ {-\lambda}, \quad j_\lambda (a_t,0)= \cosh (t)^{-\lambda}.\]
The condition \eqref{eq:HCcond} is equivalent to $\lambda>1$.
\end{example}

\section{Fock space representations}\label{sec:FockSpaceReps}

We recall the construction of the irreducible unitary representations of the nilpotent group $N$. Recall from Section~\ref{sec:P+KNgeneralCase} that $\fn=\fn_{\frac{1}{2}}\oplus\fn_1$, where
$$ \fn_{\frac{1}{2}} = \bigoplus_{j=1}^r \fg^{\frac{1}{2}\lambda_j}, \qquad \mbox{and} \qquad \fn_1 = \bigoplus_{1\leq j\leq k\leq r}\fg^{\frac{1}{2}(\lambda_j+\lambda_k)} $$
with $[\fn_{\frac{1}{2}},\fn_{\frac{1}{2}}]=\fn_1$ and $[\fn_{\frac{1}{2}}\oplus \fn_1,\fn]=\{0\}$. Hence, $\fn$ is abelian if $\calD$ is of tube-type and two-step nilpotent otherwise. In particular, $\fn_1$ is the center of $\fn$. It follows that for every non-trivial central character $\psi:N\to S^1$ there exists a unique irreducible unitary representation $\omega_\psi$ of $N$ such that $\omega_\psi(n)=\psi(n)\cdot\id$ for all $n\in\exp(\fn_1)$. If $\fn$ is abelian, then $\omega_\psi=\psi$ is simply a unitary character of $N$, and if $\fn$ is two-step nilpotent the representation $\omega_\psi$ is infinite-dimensional.

We realize $\omega_\psi$ explicitly on a Fock space $\calF$ of holomorphic functions on $\fn_{\frac{1}{2}}$, with the convention that $\calF=\C$ if $\fn_{\frac{1}{2}}=\{0\}$. More precisely, let $J$ be the complex structure on $\fn_{\frac{1}{2}}$ given by Lemma~\ref{lem:ComplexStructure}. The Killing form $\kappa$ on $\fg$ gives rise to an inner product on the real vector space $\fn$:
$$ (z|w) = \kappa(z,\theta w) \qquad (z,w\in\fn). $$
Using the complex structure $J$, we define a Hermitian form on $\fn_{\frac{1}{2}}$:
$$ \langle z,w\rangle = \kappa(z,\theta w)-i\kappa(Jz,\theta w) \qquad (z,w\in\fn_{\frac{1}{2}}). $$
It is easy to see that this form is indeed Hermitian with respect to the complex structure induced by $J$. The Fock space $\calF$ is defined by
$$ \calF = \left\{\zeta\in\calO(\fn_{\frac{1}{2}}):\int_{\fn_{\frac{1}{2}}}|\zeta(w)|^2e^{-|w|^2}\,dw<\infty\right\} $$
with the obvious norm, where $dw$ denotes Lebesgue measure on $\fn_{\frac{1}{2}}$ normalized such that $e^{-|w|^2}\,dw$ is a probability measure. $\calF$ is a reproducing kernel Hilbert space with reproducing kernel given by
\begin{equation}
	K_\calF(z,w) = e^{\langle z,w\rangle} \qquad (z,w\in\fn_{\frac{1}{2}}).\label{eq:RepKernelFock}
\end{equation}

The group $N$ acts on $\calF$ by
\begin{align}
	\omega(\exp(z))\zeta(w) &= e^{\langle w,z\rangle-\frac{1}{2}|z|^2}\zeta(w-z) && (z\in\fn_{\frac{1}{2}}),\label{eq:FockActionN1/2}\\
	\omega(\exp(x))\zeta(w) &= e^{2i(x|E)}\zeta(w) && (x\in\fn_1),\label{eq:FockActionN1}
\end{align}
where $\zeta\in\calF$ and $w\in\fn_{\frac{1}{2}}$ and $E$ as in \eqref{eq:DefE}. Here, we normalize the Killing form such that $(E|E)=r$. Note that $K\cap L$ leaves $E$ invariant.

On the dense subspace $\calF_0$ spanned by functions of the form
$$ \zeta(z) = p(z)e^{\langle z,w\rangle} \qquad (z\in\fn_{\frac{1}{2}}), $$
where $p$ is a polynomial and $w\in\fn_{\frac{1}{2}}$, this representation extends to the complexification $N_\C$ of $N$. Recall that
$$ \fn_{\frac{1}{2},\C}=\fn_{\frac{1}{2},\C}^+\oplus\fn_{\frac{1}{2},\C}^- \qquad \mbox{with }\fn_{\frac{1}{2},\C}^\pm=\{\tfrac{1}{2}(z\mp iJz):z\in\fn_{\frac{1}{2}}\} $$
and put
\begin{equation}
	n_z^\pm=\exp\left(\frac{1}{2}(z\mp iJz)\right)\in\exp(\fn_{\frac{1}{2},\C}^\pm) \qquad (z\in\fn_{\frac{1}{2}}).\label{def:nzpm}
\end{equation}
It follows that for $\zeta\in\calF_0$ and $z,w\in\fn_{\frac{1}{2}}$:
\begin{align}
	\omega(n_z^+)\zeta(w) &= \zeta(w-z),\label{eq:OscillatorActionOnNC+}\\
	\omega(n_z^-)\zeta(w) &= e^{\langle w,z\rangle}\zeta(w).\label{eq:OscillatorActionOnNC-}
\end{align}
In particular,
\begin{equation}
	\omega(n_z^+)^*=\omega(n_z^-)^{-1} \qquad (z\in\fn_{\frac{1}{2}}).\label{eq:AdjointFockSpaceRep}
\end{equation}

The compact group $K\cap L$ acts on $\fn_{\frac{1}{2}}$, $\fn_{\frac{1}{2},\C}$ and $\fn_{\frac{1}{2},\C}^\pm$ by the adjoint representation. We have
$$ \omega(n_{\Ad(k)z}^\pm) = \tau(k)\circ\omega(n_z^\pm)\circ\tau(k)^{-1} \qquad (k\in K\cap L,z\in\fn_{\frac{1}{2}}), $$
where $\tau$ is the unitary representation of $K\cap L$ on $\calF$ given by
$$ \tau(k)\zeta(z) = \zeta(\Ad(k)^{-1}z) \qquad (k\in K\cap L,\zeta\in\calF,z\in\fn_{\frac{1}{2}}). $$

\section{Whittaker models for holomorphic discrete series}\label{sec:WhittakerModels}

We fix the irreducible unitary representation $(\omega,\calF)$ of $N$ constructed in the previous section. For a holomorphic discrete series representation $(U_\pi,\calH_\pi)$ we study three types of objects:
\begin{itemize}
	\item (Whittaker vectors) $W\in\Hom_N(\calH_\pi^\infty,\calF)$
	\item (Intertwining operators) $T\in\Hom_G(\calH_\pi^\infty,C^\infty(G/N,\omega))$
	\item (Kernels) $\Psi:G\times\calD\to\Hom(\calF,\calV_\pi)$
\end{itemize}
Using explicit kernels, we embed holomorphic discrete series representations $(U_\pi,\calH_\pi)$ into $L^2(G/N,\omega)$, the space of $L^2$-sections of the homogeneous vector bundle $G\times_N\calF$ over $G/N$ induced from $(\omega,\calF)$.

\subsection{Whittaker vectors vs. intertwining operators}\label{sec:WhittakerVsIntertwiner}

Let $W\in\Hom_N(\calH_\pi^\infty,\calF)$ be a Whittaker functional. Then we obtain a continuous $G$-equivariant embedding
\begin{gather*}
	T:\calH_\pi^\infty\to C^\infty(G/N,\omega)=\{f\in C^\infty(G,\calF):f(gn)=\omega(n)^{-1}f(g)\mbox{ for }g\in G,n\in N\},\\
	T(F)(x) = W (U_\pi(x)^{-1}F ) \qquad (F\in\calH_\pi^\infty,x\in G).
\end{gather*}
On the other hand, every continuous $G$-equivariant linear operator $T:\calH_\pi^\infty\to C^\infty(G/N,\omega)$ gives rise to a Whittaker functional $W\in\Hom_N(\calH_\pi^\infty,\calF)$ by
\begin{equation}
	W(F) = T(F)(e) \qquad (F\in\calH_\pi^\infty).\label{eq:WhittakerFromIntertwiner}
\end{equation}
This gives the Frobenius reciprocity isomorphism
$$ \Hom_N(\calH_\pi^\infty,\calF) \simeq \Hom_G(\calH_\pi^\infty,C^\infty(G/N,\omega)), \quad W\mapsto T. $$

\subsection{From intertwining operators to kernels}\label{sec:IntertwinerToKernel}

To every $T\in\Hom_G(\calH_\pi^\infty,C^\infty(G/N,\omega))$ we can now associate a kernel function $\Psi:G\times\calD\to\Hom(\calF,\calV_\pi)$ in the following way. For every $z\in\calD$ and $\xi\in\calV_\pi$, the reproducing kernel $K_z\xi$ is contained in $\calH_\pi^\infty$, and we can put
$$ \Psi^*(x,z)\xi = T(K_{z,\xi})(x) = T(U_\pi(x)^*K_{z,\xi})(e) \qquad (x\in G,z\in\calD,\xi\in\calV_\pi). $$
This defines a map $\Psi^*:G\times\calD\to\Hom(\calV_\pi,\calF)$. Since $K_{z,\xi}$ is antiholomorphic in $z\in\calD$, the function $\Psi^*(x,z)$ is continuous in $x\in G$ and antiholomorphic in $z\in\calD$. Taking the adjoint of $\Psi^*(x,z)$, we obtain
$$ \Psi:G\times\calD\to\Hom(\calF,\calV_\pi), \quad \Psi(x,z) := \Psi^*(x,z)^*. $$

\begin{lemma}\label{lem:KernelEquivariance}
	The function $\Psi(x,z)$ is continuous in $x\in G$ and holomorphic in $z\in\calD$. Moreover, it has the following invariance:
	\begin{enumerate}
		\item $\Psi(xn,z)=\Psi(x,z)\circ\omega(n)$ for $n\in N$.
		\item $\Psi(gx,z)=j_\pi(g^{-1},z)\circ\Psi(x,g^{-1}z)$ for all $g,x\in G$ and $z\in\calD$.
	\end{enumerate}
\end{lemma}

\begin{proof}
	The proof is the same as in \cite[Lemma 4.4]{FOO24}.
\end{proof}

By the previous lemma, the kernel $\Psi(x,z)$ is uniquely determined by its values at $(x,o)$ for $x\in G$. We therefore let
$$ \widetilde{\Psi}:G\to\Hom(\calF,\calV_\pi), \quad \widetilde\Psi(x)=\Psi(x,o), $$
noting that
\begin{equation}
	\Psi(x,g\cdot o) = j_\pi(g,o)^{-1}\circ\widetilde\Psi(g^{-1}x) \qquad (x,g\in G).\label{eq:PsiVsTildePsi}
\end{equation}
Using the equivariance properties from before, we can express the values of $\widetilde\Psi$ in terms of the decomposition $G\subseteq P^-K_\C N_\C$ (see Remark~\ref{rem:P-Decomposition}).

\begin{lemma}\label{lem:FormulaForPsiTilde}
	Let $g\in G$ and write $g=p^-kn\in P^-K_\C N_\C$. Then the following identity holds in $\Hom(\calF_0,\calV_\pi)$ with $\calF_0$ as in Section~\ref{sec:FockSpaceReps}:
	$$ \widetilde\Psi(g) = \pi(k)\circ\widetilde\Psi(e)\circ\omega(n). $$
\end{lemma}

\begin{proof}
	The same as in \cite[Remark 4.5~(2)]{FOO24}.
\end{proof}

It follows that $\widetilde\Psi$ is uniquely determined by $\widetilde{\Psi}(e)=\Psi(e,o)\in\Hom(\calF,\calV_\pi)$. The latter space is infinite-dimensional in the case where $\calD$ is of non-tube type. However, by Lemma~\ref{lem:FormulaForPsiTilde} there is some restriction on $\widetilde{\Psi}(e)$: for $n\in N_\C\cap(K_\C P^-)$ we have
$$ \widetilde\Psi(e)\circ\omega(n) = \widetilde\Psi(n) = \pi(n)\circ\widetilde\Psi(e), $$
where we view $(\pi,\calV_\pi)$ as a representation of $K_\C P^-$ by extending it trivially to $P^-$. It follows that $A=\widetilde\Psi(e)$ belongs to $\Hom_{N_\C\cap(K_\C P^-)}(\calF_0,\calV_\pi)$.

\begin{lemma}\label{lem:IntertwinerFockSpaceToLKT}
	Every $A\in\Hom_{N_\C\cap(K_\C P^-)}(\calF_0,\calV_\pi)$ is of the form
	$$ A\zeta = \int_{\fn_{\frac{1}{2}}}\zeta(w)\pi(n_w^-)\eta \cdot e^{-|w|^2}\,dw \qquad (\zeta\in\calF) $$
	for some $\eta\in\calV_\pi$.
\end{lemma}

\begin{proof}
Taking adjoints, we have
$$ \omega(n)^*\circ A^*=A^*\circ \pi(n)^* \qquad (n\in N_\C\cap(K_\C P^-)). $$
Since $N_\C\cap K_\C P^-=\{n_z^-:z\in\fn_{\frac{1}{2}}\}$ (see Lemma~\ref{lem:ComplexStructure} and \eqref{def:nzpm}), using \eqref{eq:RepKernelFock} and \eqref{eq:OscillatorActionOnNC+} we find
\begin{align*}
	A^*\xi(z) &= [\omega(n_z^+)^{-1}\circ A^*\xi](0) = [\omega(n_z^-)^*\circ A^*\xi](0)\\
	&= [A^*\circ\pi(n_z^-)^*\xi](0) = \langle A^*\circ\pi(n_z^-)^*\xi,\mathbf{1}\rangle = \langle\xi,\pi(n_z^-)\circ A\mathbf{1}\rangle.
\end{align*}
Writing $\eta=A\mathbf{1}$ it follows that
\begin{equation*}
	\langle A\zeta,\xi\rangle = \langle\zeta,A^*\xi\rangle = \int_{\fn_{\frac{1}{2}}}\zeta(w)\langle\pi(n_z^-)\eta,\xi\rangle e^{-|w|^2}\,dw
\end{equation*}
and the claim follows.
\end{proof}

\begin{corollary}\label{cor:UpperBoundWhittaker}
	$\widetilde{\Psi}(e)\in\Hom(\calF_0,\calV_\pi)$ is given by
	$$ \widetilde{\Psi}(e)\zeta = \int_{\fn_{\frac{1}{2}}}\zeta(w)\pi(n_w^-)\eta \cdot e^{-|w|^2}\,dw \qquad (\zeta\in\calF) $$
	for some $\eta\in\calV_\pi$. In particular, the space $\Hom_N(\calH_\pi^\infty,\calF)$ of Whittaker vectors on $\calH_\pi$ is of dimension at most $\dim\calV_\pi$.
\end{corollary}

To show that the dimension is actually equal to $\dim\calV_\pi$, we construct in Section~\ref{sec:ExplicitKernels} kernel functions $\Psi$ explicitly and show that they give rise to Whittaker vectors.

\subsection{The case of unitary intertwining operators}\label{sec:UnitaryIntertwiners}

We assume for a moment that $T$ extends to an embedding $T:\calH_\pi\to L^2(G/N,\omega)$. Then its adjoint $T^*:L^2(G/N,\omega)\to\calH_\pi$ is an integral operator with integral kernel $\Psi(x,z)$:

\begin{proposition}\label{prop:AdjointIntertwiner}
	If $T:\calH_\pi\to L^2(G/N,\omega)$ is a $G$-equivariant continuous linear map and $\Psi$ as in Section~\ref{sec:IntertwinerToKernel}, then its adjoint $T^*:L^2(G/N,\omega)\to\calH_\pi$ is given by
	$$ T^*f(z) = \int_{G/N} \Psi(x,z)f(x)\,d(xN) \qquad (f\in L^2(G/N,\omega),z\in\calD). $$
\end{proposition}

\begin{proof}
	For $f\in L^2(G/N,\omega)$ we have
	\begin{align*}
		\langle T^*f(z),\xi\rangle &= \langle T^*f,K_{z,\xi}\rangle = \langle f,T(K_{z,\xi})\rangle\\
		&= \int_{G/N}\langle f(x),\Psi(x,z)^*\xi\rangle\,d(xN)\\
		&= \int_{G/N}\langle\Psi(x,z)f(x),\xi\rangle\,d(xN).\qedhere
	\end{align*}
\end{proof}

\begin{corollary}\label{cor:IntertwinerInTermsOfKernel}
	If $T:\calH_\pi\to L^2(G/N,\omega)$ is a $G$-equivariant continuous linear map and $\Psi$ as in Section~\ref{sec:IntertwinerToKernel}, then
	$$ TF(x) = \int_{\calD} \Psi(x,z)^*K_\pi(z,z)^{-1}F(z)\,d^*z \qquad (F\in \calH_\pi,x\in G). $$
\end{corollary}

\begin{proof}
	For $f\in L^2(G/N,\omega)$ we have, in view of the previous proposition:
	\begin{align*}
		\langle TF,f\rangle &= \langle F,T^*f\rangle = \int_{\calD}\langle K_\pi(z,z)^{-1}F(z),T^*f(z)\rangle\,d^*z\\
		&= \int_{\calD}\int_{G/N}\langle K_\pi(z,z)^{-1}F(z),\Psi(x,z)f(x)\rangle\,d(xN)\,d^*z\\
		&= \int_{G/N}\left\langle\int_{\calD}\Psi(x,z)^*K_\pi(z,z)^{-1}F(z)d^*z,f(x)\right\rangle\,d(xN).\qedhere
	\end{align*}
\end{proof}

\subsection{From kernels to unitary intertwining operators}\label{sec:ExplicitKernels}

Now let $\eta\in\calV_\pi$ and define $A_\eta\in\Hom(\calF,\calV_\pi)$ by
$$ A_\eta\zeta = \int_{\fn_{\frac{1}{2}}}\zeta(w)\pi(n_w^-)\eta \cdot e^{-|w|^2}\,dw \qquad (\zeta\in\calF). $$
Note that since $\pi(n_w^-)\eta$ is an antiholomorphic polynomial in $w$, the integral converges for all $\zeta\in\calF$. In view of the proof of Lemma~\ref{lem:IntertwinerFockSpaceToLKT}, the adjoint $A_\eta^*\in\Hom(\calV_\pi,\calF)$ is given by
$$ A_\eta^*\xi(z) = \langle\xi,\pi(n_z^-)\circ\eta\rangle = \langle\pi(n_z^+)^{-1}\xi,\eta\rangle \qquad (\xi\in\calV_\pi,z\in\fn_{\frac{1}{2}}). $$
We define $\widetilde\Psi_{\pi,\eta}:G\to\Hom(\calF_0,\calV_\pi)$ at $s=p^-_N(s)k_{\C,N}^-(s)n_\C^-(s)\in P^-K_\C N_\C$ by
$$ \widetilde\Psi_{\pi,\eta}(s) = \pi(k_{\C,N}^-(s))\circ A_\eta\circ\omega(n_\C^-(s)). $$
By the $(N_\C\cap K_\C P^-)$-equivariance of $A_\eta$, this is well-defined (cf. Remark~\ref{rem:DisambiguityDecomposition}) and gives rise to a kernel function $\Psi_{\pi,\eta}:G\times\calD\to\Hom(\calF_0,\calV_\pi)$ by \eqref{eq:PsiVsTildePsi}. For later purpose, we note the following formula for the adjoint $\Psi_{\pi,\eta}^*:G\times\calD\to\Hom(\calV_\pi,\calF)$ which follows from \eqref{eq:Decomposition+vs-}, \eqref{eq:PiAdjoint} and \eqref{eq:AdjointFockSpaceRep}:
\begin{align}
	\Psi_{\pi,\eta}(x,g\cdot o)^* &= \widetilde\Psi_{\pi,\eta}(g^{-1}x)^*\circ j_\pi(g,o)^{-*}\notag\\
	&= \omega(n_\C^-(g^{-1}x))^*\circ A_\eta^*\circ\pi(k_{\C,N}^-(g^{-1}x))^*\circ j_\pi(g,o)^{-*}\notag\\
	&= \omega(n_\C^+(g^{-1}x))^{-1}\circ A_\eta^*\circ\pi(k_{\C,N}^+(g^{-1}x))^{-1}\circ j_\pi(g,o)^{-*}.\label{eq:FormulaPsiStar}
\end{align}
Using the notation
\begin{equation}
	m:\calV_\pi\otimes\overline{\calV}_\pi\to\calF, \quad m(\xi\otimes\overline{\eta})(z) = \langle\pi(n_z^+)^{-1}\xi,\eta\rangle \qquad (\xi,\eta\in\calV_\pi,z\in\fn_{\frac{1}{2}}),\label{eq:DefMatrixCoeffMap}
\end{equation}
so that $m(\xi\otimes\overline{\eta})$ is a polynomial and hence contained in $\calF_0$, this can be written as
\begin{equation}
	\Psi_{\pi.,\eta}(x,g\cdot o)^*\xi = \omega(n_\C^+(g^{-1}x))^{-1}m\Big(\pi(k_{\C,N}^+(g^{-1}x))^{-1}\circ j_\pi(g,o)^{-*}\xi\otimes\overline{\eta}\Big).
\end{equation}
In particular, $\Psi_{\pi,\eta}(x,z)^*\in\Hom(\calV_\pi,\calF)$ so $\Psi_{\pi,\eta}(x,z)\in\Hom(\calF,\calV_\pi)$ instead of $\Hom(\calF_0,\calV_\pi)$.

In view of Corollary~\ref{cor:IntertwinerInTermsOfKernel}, we define (whenever the integral converges)
$$ T_{\pi,\eta}:\calH_\pi^\infty \to C^\infty(G/N,\omega), \quad T_{\pi,\eta}F(x) = \int_{\calD}\Psi_{\pi,\eta}(x,z)^*K_\pi(z,z)^{-1}F(z)\,d^*z. $$

In order to show that $T_{\pi,\eta}$ defines an intertwining operator $\calH_\pi\to L^2(G/N,\omega)$, it suffices to show that $T_{\pi,\eta}F\in L^2(G/N,\omega)$ for $F$ contained in the lowest $K$-type. Recall that the lowest $K$-type of $\calH_\pi$ is spanned by the functions $z\mapsto K_\pi(z,o)\xi=\xi$ for $\xi\in\calV_\pi$.

\begin{lemma}\label{lem:LKTembedding}
	Identifying $\xi\in\calV_\pi$ with the constant function $z\mapsto\xi$, the restriction of $T_{\pi,\eta}$ to $\calV_\pi$ is given by
	$$ T_{\pi,\eta}\xi(x) = \Psi_{\pi,\eta}(x,o)^*\xi = \omega(n_\C^+(x))^{-1}\circ A_\eta^*\circ\pi(k_{\C,N}^+(x))^{-1}\xi \qquad (\xi\in\calV_\pi,x\in G). $$
\end{lemma}

\begin{proof}
	The first identity follows from the reproducing kernel property of $K_\pi$ and the second identity is a special case of \eqref{eq:FormulaPsiStar}.
\end{proof}

In particular, for $x=k_1ak_2$ with $k_1\in K$, $a\in A$ and $k_2\in K\cap L$ we find
\begin{equation*}
	T_{\pi,\eta}\xi(x) = \omega(k_2^{-1}n_\C^+(a)k_2)^{-1}\circ A_\eta^*\circ\pi(k_1k_{\C,N}^+(a)k_2)^{-1}\xi.
\end{equation*}
By $\SU(1,1)$-reduction, $n_\C^-(a)$ is contained in the center of $N_\C$ for $a\in A$. Since $K\cap L$ leaves $E\in\fn_1$ invariant, it follows from \eqref{eq:FockActionN1} that $\omega(k_2^{-1}n_\C^+(a)k_2)=\omega(n_\C^+(a))$. This implies:
\begin{equation*}
	T_{\pi,\eta}\xi(x) = \omega(n_\C^+(a))^{-1}\circ A_\eta^*\circ\pi(k_2)^{-1}\circ\pi(k_{\C,N}^+(a))^{-1}\circ\pi(k_1)^{-1}\xi.
\end{equation*}

\begin{theorem}\label{thm:EmbeddingsAreL2}
	$T_{\pi,\eta}\xi\in L^2(G/N,\omega)$ for some resp. all $\xi\in\calV_\pi$ if and only if $\calH_\pi$ belongs to the holomorphic discrete series.
\end{theorem}

\begin{proof}
	We use the integral formulas
	$$ \int_{G/N}\varphi(x)\,d(xN) = \int_K\int_L\varphi(kl)l^{2\rho_\fn}\,dl\,dk $$
	and
	$$ \int_L\varphi(l)\,dl = \int_{K\cap L}\int_{K\cap L}\int_A\varphi(k_1ak_2)w(a)\,da\,dk_1\,dk_2, $$
	where
	$$ w(\exp H) = \prod_{\alpha\in\Delta^+(\fl,\fa)}(\sinh\alpha(H))^{\dim\fl^\alpha}. $$
	With this, we find
	\begin{multline*}
		\int_{G/N}\|T_{\pi,\eta}\xi\|^2\,d(gN) = \int_K\int_A\int_{K\cap L}\|\omega(n_\C^+(a))^{-1}\circ\\
		A_\eta^*\circ\pi(k_2)^{-1}\circ\pi(k_{\C,N}^+(a))^{-1}\circ\pi(k_1)^{-1}\xi\|^2\,a^{2\rho_n}w(a)\,dk_2\,da\,dk_1.
	\end{multline*}
	By Schur's orthogonality relations this equals up to a constant
	\begin{equation}
		\int_A\|\omega(n_\C^+(a))^{-1}\circ A_\eta^*\circ\pi(k_{\C,N}^+(a))^{-1}\xi\|^2 a^{2\rho_{\fn}}w(a)\,da.\label{eq:HopefullyConvergentIntegral}
	\end{equation}
	Now, by the $\SU(1,1)$-reduction in Corollary~\ref{thm:pkn2} we find for $a=\exp\sum_{j=1}^r t_jx_j$:
	\begin{align*}
		k_{\C,N}^+(a) &= \exp\left(-\sum_{j=1}^r t_jh_j\right),\\
		n_\C^+(a) &= \exp\left(\frac{1}{2i}\sum_{j=1}^r(1-e^{-2t_j})\varphi_j \begin{pmatrix} i & - i\\ i & - i \end{pmatrix}\right).
	\end{align*}
	By \eqref{eq:FockActionN1} we have
	$$ \omega(n_\C^+(a))^{-1} = \exp\left(-\sum_{j=1}^r(1-e^{-2t_j})\right). $$
	Moreover, if $\xi\in\calV_\pi$ is a vector in the highest restricted weight space, then by \eqref{eq:ActionHighestWeightVector}:
	$$ \pi(k_{\C,N}^+(a))^{-1}\xi = \left(\prod_{j=1}^re^{-\mu_{\pi,j}t_j}\right)\xi. $$
	It follows that \eqref{eq:HopefullyConvergentIntegral} equals
	$$ \int_{t_1>\cdots>t_r} \prod_{j=1}^r\exp\left(-2(1-e^{-2t_j})-2\mu_{\pi,j}t_j+2\rho_{\fn,j}t_j\right)\cdot w\left(\exp\sum_{j=1}^rt_jx_j\right)\,dt_1\cdots dt_r. $$
	Because of the superexponential decay as $t_j\to-\infty$, it is sufficient to study the integral over $t_1>\cdots>t_r>0$. In this region, we can estimate $w(a)\sim a^{2\rho_\fl}$, so the integral under consideration becomes
	$$ \int_{t_1>\cdots>t_r>0} \prod_{j=1}^r\exp\left(-2(1-e^{-2t_j})-2\mu_{\pi,j}t_j+2\rho_{\fn,j}t_j+2\rho_{\fl,j}t_j\right)\,dt_1\cdots dt_r. $$
	This integral clearly converges if and only if $\mu_{\pi,j}>\rho_{\fn,j}+\rho_{\fl,j}=\rho_j$ for all $1\leq j\leq r$, which is precisely condition \eqref{eq:DSconditionMuJ}. Finally, $T_{\pi,\eta}\xi\in L^2(G/N,\omega)$ holds for $\xi$ in the highest restricted weight space if and only if it holds for all $\xi\in\calV_\pi$ by irreducibility.
\end{proof}

This shows that every $T_{\pi,\eta}$ extends to $\calH_\pi\to L^2(G/N,\omega)$. Because $\calH_\pi$ is irreducible and unitary, this extension must be unitary (up to a scalar). As shown in Section~\ref{sec:UnitaryIntertwiners}, every unitary intertwining operator $\calH_\pi\to L^2(G/N,\omega)$ must be of this form, so we have shown:

\begin{corollary}\label{cor:HolDScontribution}
	For every $\eta\in\calV_\pi$ the operator $T_{\pi,\eta}$ extends to a $G$-equivariant unitary (up to a scalar) operator
	$$ T_{\pi,\eta}:\calH_\pi\to L^2(G/N,\omega). $$
	Moreover, the map
	$$ \bigoplus_\pi\calH_\pi\otimes\overline{\calV_\pi}\to L^2(G/N,\omega), \quad \calH_\pi\otimes\overline{\calV_\pi}\ni F\otimes\overline\eta\mapsto T_{\pi,\eta}F $$
	is an embedding onto the holomorphic discrete series contribution to the Plancherel decomposition of $L^2(G/N,\omega)$.
\end{corollary}

\subsection{Back to Whittaker vectors}\label{sec:BackToWhittaker}

Every unitary intertwining operator $T_{\pi,\eta}:\calH_\pi\to L^2(G/N,\omega)$ constructed in the previous section restricts to an operator $\calH_\pi^\infty\to C^\infty(G/N,\omega)$, so the inequality in Corollary~\ref{cor:UpperBoundWhittaker} is actually an equality. We now identify the corresponding Whittaker vectors with $\calF$-valued antiholomorphic functions on $\calD$.

As in \cite[Lemma 2.4]{FOO24}, we identify a Whittaker vector $W\in\Hom_N(\calH_\pi^\infty,\calF)$ with an antiholomorphic function $\Pi:\calD\to\calV_\pi^\vee\otimes\calF$ in the sense that
$$ W(F) = \int_\calD\Pi(z)\left(K_\pi(z,z)^{-1}F(z)\right)\,d^*z \qquad (F\in\calH_\pi^\infty). $$
Let $\Pi_{\pi,\eta}$ denote the function associated with the Whittaker vector which arises from the intertwiner $T_{\pi,\eta}$ (in the sense of Section~\ref{sec:WhittakerVsIntertwiner}), i.e.
\begin{equation}
	T_{\pi,\eta}(F)(e) = \int_\calD\Pi_{\pi,\eta}(z)\left(K_\pi(z,z)^{-1}F(z)\right)\,d^*z \qquad (F\in\calH_\pi^\infty).\label{eq:AntiholWhittakerVectorInTermsOfIntertwiner}
\end{equation}

\begin{proposition}\label{prop:WhittakerVectorsAsAntiholFcts}
	As an element of $\Hom(\calV_\pi,\calF)\simeq\calV_\pi^\vee\otimes\calF$, the following identity holds, where $z=g\cdot o$:
	$$ \Pi_{\pi,\eta}(z) = \Psi(e,z)^* = \omega(n_\C^+(g^{-1}))^{-1}\circ A_\eta^*\circ\pi(k_{\C,N}^+(g^{-1}))^{-1}\circ j_\pi(g,o)^{-*}. $$
\end{proposition}

\begin{proof}
	The first identity is immediately clear from \eqref{eq:WhittakerFromIntertwiner}, Corollary~\ref{cor:IntertwinerInTermsOfKernel} and \eqref{eq:AntiholWhittakerVectorInTermsOfIntertwiner}. The second one is a special case of \eqref{eq:FormulaPsiStar}.
\end{proof}


\providecommand{\bysame}{\leavevmode\hbox to3em{\hrulefill}\thinspace}
\providecommand{\MR}{\relax\ifhmode\unskip\space\fi MR }
\providecommand{\MRhref}[2]{%
	\href{http://www.ams.org/mathscinet-getitem?mr=#1}{#2}
}
\providecommand{\href}[2]{#2}

\end{document}